\documentclass[a4paper,12pt]{amsart}

\usepackage[english]{babel}
\usepackage{tikz}
\usetikzlibrary{positioning,arrows,cd}

\usepackage{amsmath,amssymb,amsthm}

\usepackage{mathdots}

\usepackage{enumitem}

\usepackage{verbatim}

\usepackage{xcolor}

\newtheorem{theorem}{Theorem}[section]
\newtheorem{proposition}[theorem]{Proposition}
\newtheorem{lemma}[theorem]{Lemma}
\newtheorem{corollary}[theorem]{Corollary}

\newtheorem{question}[theorem]{Question}

\theoremstyle{definition}
\newtheorem{definition}[theorem]{Definition}

\theoremstyle{remark}
\newtheorem{example}[theorem]{Example}				
\newtheorem{remark}[theorem]{Remark}

\numberwithin{equation}{section}

\newcommand{\Ann}{\operatorname{Ann}}
\newcommand{\Aut}{\operatorname{Aut}}

\newcommand{\car}{\operatorname{char}}
\renewcommand{\wr}{\operatorname{wr}}
\newcommand{\sr}{\operatorname{sr}}

\newcommand{\F}{{\mathbb F}}

\usepackage[pdfpagelabels]{hyperref}

\usepackage{colortbl}
\DefineNamedColor{named}{RoyalBlue}     {cmyk}{1,0.50,0,0}
\DefineNamedColor{named}{BrickRed}      {cmyk}{0,0.89,0.94,0.28}

\AtBeginDocument{%
   \def\MR#1{}
}

\begin{document}

\title{Commutative algebras are ideals of evolution algebras}
\author[C.\ Costoya]{Cristina Costoya}
\address{CITMAga, Departamento de Matem\'aticas,
Universidade de Santiago de Compostela, 15782-Santiago de Compostela, Spain.}
\email{cristina.costoya@usc.es}
\author[A. Fernández Ouaridi]{Amir Fernández Ouaridi}
\address{Departamento de Geometría y Topología, Universidad de Sevilla, Sevilla, Spain.}
\email{amir.fernandez.ouaridi@gmail.com}

\author{Antonio Viruel}
\address{Departamento de \'Algebra, Geometr{\'\i}a y Topolog{\'\i}a, Universidad de M{\'a}laga, 29071-M{\'a}laga, Spain}
\email{viruel@uma.es}
\thanks{Authors were partially supported by grant PID2023-149804NB-I00 funded by MCIN/AEI/
10.13039/501100011033.}


\begin{abstract}
We prove that every finite-dimensional commutative algebra can be embedded as an ideal of a finite-dimensional evolution algebra. We then determine the smallest possible dimension of an evolution algebra admitting such an embedding and study some structural properties preserved by the construction. Over fields of characteristic different from two, we relate this minimal dimension to simultaneous Waring decompositions of the quadratic map $x\mapsto x^2$. In particular, for unital commutative algebras, the smallest possible dimension of an evolution algebra containing the given algebra coincides with the Waring rank of its squaring map. As an application, we determine this minimal dimension for the truncated polynomial algebras $\mathbb{F}[t]/(t^n)$. Finally, we study embeddings for which additional structure is preserved, including automorphisms and multiplicative bases.
\end{abstract}

\maketitle

\section{Introduction}

Evolution algebras were introduced by Tian and Vojt\v{e}chovský in connection with non-Mendelian genetics \cite{TV06}, and their systematic algebraic study was subsequently developed by Tian \cite{Tian}. Since then, they have attracted considerable attention from a purely algebraic perspective, leading to a substantial structure theory involving, among other topics, nilpotency \cite{EL15, EL16}, subalgebras, ideals and decompositions \cite{BCS22, CSV16}, automorphisms \cite{EL19, chinos,CLTV22},  and classification \cite{CFK24, Class2dim}; see also \cite{idrees, GP25}.

The defining feature of an evolution algebra is the existence of a particularly simple type of basis. More precisely, an algebra $E$ over a field $\mathbb{F}$ is called an \emph{evolution algebra} if it admits a basis $\mathcal{B}=\left\{e_i\;\, i\in I\right\}$, called a \emph{natural basis}, such that
$$
e_i e_j=0 \qquad \text{whenever } i\neq j.
$$
Evolution algebras are therefore commutative, but in general not associative. Also, relative to a natural basis, the multiplication is completely determined by the squares of the basis elements. Despite the simplicity of this defining condition, evolution algebras exhibit a rich algebraic structure, and the distinguished role of a natural basis allows algebraic properties to be studied through linear algebra and combinatorial methods, in particular through the associated directed graph.

The tension between the rigidity imposed by a natural basis and the class of algebraic structures that may occur inside an evolution algebra is central to the present work. On the one hand, determining whether a given finite-dimensional commutative algebra is itself an evolution algebra is a restrictive recognition problem, which can be formulated in terms of simultaneous diagonalization by congruence of a family of symmetric matrices associated with the multiplication \cite{BMV20}. On the other hand, subalgebras of evolution algebras need not themselves be evolution algebras \cite{CSV16}, so passing to subalgebras allows substantially more general commutative algebraic structures to appear.  This behaviour may change significantly under additional assumptions on the ambient algebra; for instance, regular evolution algebras, those satisfying $E^2=E$, are closed under taking subalgebras \cite{LP25}.

It is precisely the gap between being an evolution algebra and occurring inside one that provides the starting point for this work. Our initial observation, and the main motivation for the developments that follow, is that this gap is as large as possible in finite dimension: every finite-dimensional commutative algebra over $\mathbb{F}$ can be embedded as an ideal of a finite-dimensional evolution algebra. The proof is constructive and yields a uniform finite-dimensional realization. This universality phenomenon is somewhat surprising: although the existence of a natural basis imposes a highly constrained form on the multiplication of an evolution algebra itself, no restriction remains on the finite-dimensional commutative algebras that may occur as its ideals.

Once existence is established, the problem naturally shifts to optimality, namely to determining the smallest possible dimension of an evolution algebra admitting such an embedding. We study this dimension and, over fields of characteristic different from two, show it admits an interpretation through the squaring map
$$
x\longmapsto x^2.
$$
In this setting, the multiplication of a commutative algebra can be recovered from its squaring map by polarization; see \cite{Arenas21, UU19}. Moreover, the existence of a natural basis for an evolution algebra translates into the absence of mixed terms in the quadratic map. Precisely, the minimal embedding problem can be reformulated in terms of simultaneous Waring decompositions of this map. We show the minimal dimension of an ambient evolution algebra is determined by the Waring rank of the squaring map together with the dimension of the annihilator of the original algebra. In fact, for unital commutative algebras the minimal dimension coincides exactly with the Waring rank of the squaring map.  This provides a bridge between commutative algebras and rank problems for polynomial maps, and connects the embedding problem with classical questions in multiplicative complexity \cite{Alder-Strassen}.

The paper is organized as follows. Section~\ref{comideals} contains the main embedding theorem, where we give an explicit construction showing that every finite-dimensional commutative algebra can be embedded as an ideal of an evolution algebra. Section~\ref{minimal} studies the smallest possible dimension of an evolution algebra admitting such an embedding and establishes several structural properties of the resulting embeddings. Section~\ref{secwaring} relates this minimal dimension to simultaneous Waring decompositions of the squaring map of the algebra. Section~\ref{truncated} applies these results to the truncated polynomial algebras $\mathbb{F}[t]/(t^n)$ and determines their minimal evolution dimension. Finally, Section~\ref{addi} studies further embeddings preserving additional structure, namely automorphisms and multiplicative bases.

\section{Commutative algebras as ideals of evolution algebras}\label{comideals}

The starting point and original motivation for this work is the observation that every finite-dimensional commutative algebra can be realized as an ideal of an evolution algebra. We first give a direct and completely explicit construction establishing this fact.

\begin{theorem}\label{thm:main}
	Let $A$ be a commutative $\mathbb{F}$-algebra, $n:=\dim(A)<\infty$. Then there exists $X$, an evolution $\mathbb{F}$-algebra of dimension $\frac{n(n+1)}{2}$ and an algebra morphism $\psi\colon A\longrightarrow X,$ such that 
    \begin{enumerate}[label={\rm (\alph{*})}]
    \item\label{thm:main_i} $\psi$ is a monomorphism.
    \item\label{thm:main_ii} $\psi(A)$ is an ideal of $X$.
    \item\label{thm:main_iii} $X/\psi(A)$ is a degenerate algebra.
    \end{enumerate} 
\end{theorem}
\begin{proof}
Let $\{e_1,\ldots, e_n\}$ be a basis of $A$ with structure coefficients $\{m_{ijk}\}_{i,j,k=1}^n$ given by the formula
$$e_ie_j=\sum_{k=1}^n m_{ijk} e_k.$$
Notice that since $A$ is commutative $m_{ijk}=m_{jik}$ for every $i,j,k=1,\ldots,n.$

Now, given integers $0<i,j\leq n,$ we let $\{i,j\}$ denote the unordered pair with elements $i$ and $j$, so $\{i,j\}=\{j,i\}.$ Thus there is a total of  $\frac{n(n+1)}{2}$ unordered pairs $\{i,j\}$ for $i,j=1,\ldots, n.$

Let $X$ be the evolution $\mathbb{F}$-algebra generated by the natural basis $B=\{b_{\{i,j\}} :  i,j=1,\ldots, n\},$ thus $\dim(X)=\frac{n(n+1)}{2},$ and multiplication given by the following squares:
\begin{enumerate}[label={\rm (\roman{*})}]
    \item\label{i} If $i\ne j$ then
    \begin{equation}
        b_{\{i,j\}}^2=\sum_{k=1}^n m_{ijk}\big(\sum_{r=1}^n b_{\{k,r\}}\big).
    \end{equation}

    \item\label{ii}  If $i=j$ then
    \begin{equation}
        b_{\{i,i\}}^2=\sum_{k=1}^n m_{iik}\big(\sum_{r=1}^n b_{\{k,r\}}\big) - \sum_{\substack{j=1 \\ j\neq i}}^n b_{\{i,j\}}^2.
    \end{equation}
\end{enumerate}
Notice that $b_{\{i,i\}}^2$ is well defined since the values of $b_{\{i,j\}}^2$ have been previously defined in \ref{i}.

Define $\psi\colon A\longrightarrow X$ to be the linear monomorphism spanned by 
$$
\psi(e_i)=\sum_{j=1}^n b_{\{i,j\}}.
$$
We claim $\psi$ is a $\mathbb{F}$-algebra morphism, that is, for any given pair of generators $e_i,e_j\in A,$ it holds $\psi(e_ie_j)=\psi(e_i)\psi(e_j).$ Indeed:
\begin{enumerate}[resume,label={\rm (\roman{*})}]
\setlength\itemsep{1em}
    \item\label{iii} If $i\ne j$ then we have
    \begin{align*}
    \psi(e_ie_j) &=\psi\big(\sum_{k=1}^n m_{ijk}e_k\big)=\sum_{k=1}^n m_{ijk}\psi(e_k)\\
    &=\sum_{k=1}^n m_{ijk}\big(\sum_{r=1}^n b_{\{k,r\}}\big)= b_{\{i,j\}}^2\text{\, (by \ref{i})}, 
    \end{align*}
    and
    \begin{align*}
    \psi(e_i)\psi(e_j) &=(\sum_{k=1}^n b_{\left\{i, k\right\}})(\sum_{r=1}^n b_{\left\{j,r\right\}})\\
    &= b_{\left\{i,j\right\}}b_{\left\{j,i\right\}}\text{\, (since $i\ne j$)}= b_{\{i,j\}}^2.
    \end{align*}
    hence $\psi(e_ie_j)=\psi(e_i)\psi(e_j).$

    \item\label{iv}  If $i=j$ then we obtain
    \begin{align*}
    \psi(e_ie_i) &=\psi\big(\sum_{k=1}^n m_{iik}e_k\big)
    =\sum_{k=1}^n m_{iik}\psi(e_k)=\sum_{k=1}^n m_{iik}\big(\sum_{r=1}^n b_{\{k,r\}}\big),
    \end{align*}
    and
    \begin{align*}
    \psi(e_i)\psi(e_i) &=(\sum_{j=1}^n b_{\{i,j\}})^2=\sum_{j=1}^n b_{\{i,j\}}^2= b_{\{i,i\}}^2+\sum_{\substack{j=1 \\ j\neq i}}^n b_{\{i,j\}}^2\\
    &=\Big(\sum_{k=1}^n m_{iik}\big(\sum_{r=1}^n b_{\{k,r\}}\big) - \sum_{\substack{j=1 \\ j\neq i}}^n b_{\{i,j\}}^2\Big)+\sum_{\substack{j=1 \\ j\neq i}}^n b_{\{i,j\}}^2\\
    &=\sum_{k=1}^n m_{iik}\big(\sum_{r=1}^n b_{\{k,r\}}\big).
    \end{align*}
    hence $\psi(e_i^2)=\psi(e_i)^2.$
\end{enumerate}
Hence $\psi$ is an algebra monomorphism.

We now show that $\psi(A)$ is an ideal of $X$. Given integers $0<i,r,s\leq n,$ we get 
\begin{equation}\label{eq:ideal}
\psi(e_i)b_{\{r,s\}}=
\begin{cases}
    0, &\text{if $i\not\in \{r,s\}$}\\
    b_{\{i,j\}}^2=\psi(e_ie_j)& \text{if $\{i,j\}= \{r,s\}$ and $i\ne j$}\\
    b_{\{i,i\}}^2=\psi(e_i^2)-\sum_{\substack{j=1 \\ j\neq i}}^n \psi(e_ie_j)& \text{if $r=s=i$}.
\end{cases}    
\end{equation}
In all cases, $\psi(e_i)b_{\{r,s\}}\in\psi(A)$ for every generator $b_{\{r,s\}}$ of $X$ and $e_i$ of $A$. Thus $\psi(A)$ is an ideal (recall $X$ is commutative).

Finally $X/\psi(A)$ is indeed degenerate; if $\overline{x}\in X/\psi(A)$ denotes the image of the element $x\in X$, we obtain $\overline{b_{\{i,j\}}}^2=\overline{b_{\{i,j\}}^2}$ and Equation \eqref{eq:ideal} shows that the square of every generator of $X$ is in $\psi(A),$ hence $\overline{b_{\{i,j\}}}^2=0$.
\end{proof}

\begin{remark}
     The previous theorem covers the finite-dimensional case. For arbitrary dimension, the natural generalization does not work. The polynomial algebra $\mathbb{F}[x]$ cannot be embedded in any evolution $\mathbb{F}$-algebra with a Hamel natural basis. This is because the unit cannot be written as a finite sum of natural vectors since that would imply that the image of the multiplication operator $L_1$ has finite dimension. Thus, the following question arises.
\end{remark}

\begin{question}
     Can the embedding theorem be extended to infinite-dimensional commutative algebras in any sense? 
\end{question}

Theorem~\ref{thm:main} provides a direct existence proof together with a uniform
explicit bound. We now turn to a different question: determining the exact minimal
possible dimension of an evolution algebra embedding $A$. For this purpose we
introduce evolution envelopes.

\section{Minimal evolution envelopes}\label{minimal}

In order to determine the minimal possible dimension of an evolution algebra into which a given commutative algebra embeds, it is convenient to restrict our attention to a special class of embeddings. 

\begin{definition}
Let $A$ be a commutative algebra over $\mathbb F$. An
\emph{evolution envelope} of $A$ is a pair $(E,\Phi)$, where $E$ is an evolution $\mathbb{F}$-algebra and $\Phi\colon A\hookrightarrow E$
is a monomorphism of algebras such that $E^2\subseteq \Phi(A)$. 
\end{definition}

The following result shows that, for the purpose of determining the minimal possible dimension of an evolution algebra embedding $A$, it is enough to consider evolution envelopes.

\begin{proposition}\label{propenv}
    Let $A$ be a finite-dimensional commutative $\mathbb{F}$-algebra. Suppose there is an embedding $\phi\colon A\hookrightarrow E$ into an evolution $\mathbb{F}$-algebra $E$. Then there is an evolution envelope $(\widetilde{E}, \phi)$ of $A$ such that $\widetilde{E} = E$ as vector spaces.
\end{proposition}
\begin{proof}
    Let $\pi\colon E \rightarrow \phi(A)$ be a projection. Denote by $\widetilde{E}$ the algebra with vector space $E$ endowed with the multiplication $x * y = \pi(x \cdot_E y)$ for $x, y \in \widetilde{E}$. Clearly, $\widetilde{E}$ is an evolution algebra and we have $\widetilde{E}^2 \subseteq \phi(A)$. We verify that $\phi\colon
    A \rightarrow \widetilde{E}$ is an algebra homomorphism. For every $x, y\in A$, we have
    $$\phi(x y) =   
    \pi(\phi(xy)) = \pi(\phi(x) \cdot_E \phi(y)) = \phi(x) * \phi(y).$$
    Hence, the pair $(\widetilde{E}, \phi)$ is an evolution envelope of $A$. 
\end{proof}

We denote by
$$\operatorname{edim}(A):=\min\{\dim E:\ A\hookrightarrow E,\ E\text{ an evolution algebra}\}$$
the minimal evolution dimension of $A$. By Proposition~\ref{propenv}, the same minimum is obtained if one restricts to evolution envelopes. 

\begin{definition}
    Let $A$ be a commutative algebra of finite dimension over an arbitrary field. Given a basis $B = \{b_1, \ldots, b_n\}$, we define 
    $\varepsilon_B(A)$ as the minimal non-negative integer $s$ such that there exist $v_1, \ldots, v_s\in A$ and $\beta_{1}, \ldots, \beta_s \in A^*$ such that
    $b_i b_j = \sum_{k=1}^s \beta_k(b_i)\beta_k(b_j) v_k$ for all $i\neq j$. 
    Moreover, define   $$\varepsilon(A)=\textrm{min}\left\{\varepsilon_B(A): B \textrm{ basis of $A$}\right\}.$$ 
\end{definition}

\begin{remark}
    The index $\varepsilon_B(A)$ is well-defined, because for each $i<j$ we can take $v_{ij} = b_ib_j$ for $i\neq j$ and $\beta_{ij}\in A^*$ such that $\beta_{ij}(b_k) = \delta_{ik} + \delta_{jk}$. 
    Moreover, $\varepsilon(A)\leq n(n-1)/2$ in general, and $\varepsilon(A) = 0$ if and only if $A$ is an evolution algebra.
\end{remark}

The invariant $\varepsilon(A)$ measures how many additional dimensions are needed to absorb the cross-products. The next result shows that this
interpretation is exact.

\begin{theorem}\label{thm:embedding_amir}
    Suppose $A$ is a commutative algebra of dimension $n$ over an arbitrary field. 
    Then there is an evolution algebra $E$ of dimension $n + \varepsilon(A)$ such that $\Phi:A\hookrightarrow E$ and $E^2 \subseteq \Phi(A)$. Moreover, the algebra $E$ is a minimal-dimensional evolution algebra embedding $A$. 
\end{theorem}
\begin{proof}
    Choose $B= \left\{b_{1},\dots,b_{n}\right\}$ a basis of $A$ with $\varepsilon_B(A) = \varepsilon(A) = s$. Thus, there are $v_1, \ldots, v_s\in A$ and $\beta_{1}, \ldots, \beta_s \in A^*$ such that
    $b_i b_j = \sum_{k=1}^s \beta_k(b_i)\beta_k(b_j) v_k$ for all $i\neq j$. 
    Denote by $E$ the vector space spanned by $B_0= \left\{e_{1},\dots,e_{n}\right\} \cup \left\{f_{1}, \ldots, f_{s}\right\}$. Fix $xy=0$ for $x\neq y$ with $x,y\in B_0$. It remains to define the products of the squares of elements in $B_0$. 

    Consider the linear map $\Phi:A\to E$ given by
$$
\Phi(b_i)=e_i+\sum_{k=1}^s \beta_{k}(b_i) f_k.
$$
This map is injective, because $e_i$ only appears in the support of $\Phi(b_i)$.
We choose the squares so that $\Phi$ becomes a homomorphism of algebras:
$$
f_{k}^2=\Phi(v_k)  \textrm{ for }  1\leq k \leq s  \textrm{ and } 
e_i^2=\Phi(b_i^2)-\sum_{k=1}^s \beta_{k}(b_i)^2 \Phi(v_k)  \textrm{ for } 1\leq i\leq n.
$$
We verify that $\Phi$ is a homomorphism. For $i\ne j$, we have
$$\Phi(b_i b_j)  = \sum_{k=1}^s \beta_k(b_i)\beta_k(b_j) \Phi(v_k)  = \sum_{k=1}^s \beta_k(b_i)\beta_k(b_j) f^2_k = \Phi(b_i)\Phi(b_j).
$$
Also, for $i=j$, we obtain
$$
\Phi(b_i^2)
= e_i^2 +\sum_{k=1}^s \beta_{k}(b_i)^2 \Phi(v_k) = \Phi(b_i)^2.
$$
Therefore, the map $\Phi:A\rightarrow E$ is a monomorphism. Note that $E^2 \subseteq \Phi(A)$.

Lastly, we show that the dimension of $E$ is minimal. Suppose there is $E'$ with $\dim\, E' < \dim \, E$ such that  $\Phi':A\hookrightarrow E'$. By Proposition~\ref{propenv}, we may replace $E'$ by an evolution envelope on the same underlying vector space. Consider the coordinate matrix of the basis $\left\{\Phi'(b_1), \ldots, \Phi'(b_n)\right\}$ of $\Phi'(A)$ with respect to a natural basis $\left\{e_1', \ldots, e_m'\right\}$ of $E'$. Applying Gaussian elimination, we may obtain a reduced column echelon form matrix corresponding to a new basis $\left\{b_1', \ldots, b_n'\right\}$ of $\Phi'(A)$. By reordering the basis, we obtain a natural basis $\left\{e_{1}',\dots,e_{n}'\right\} \cup \left\{f_{1}', \ldots, f_{t}'\right\}$ 
of $E'$ with $b_i' = e_i' + \sum_{k=1}^t \beta'_{ik} f_k'$ for $1\leq i \leq n$. Hence, we have $b_i'b_j' = \sum_{k=1}^t \beta'_{ik}\beta'_{jk} (f_k')^2$ for $i\neq j$. 

Choosing $\beta_1', \ldots, \beta_t' \in \Phi'(A)^*$ with $\beta_k'(b'_i) = \beta'_{ik}$ and $v_i$ the projection of $(f_i')^2$ in $\Phi'(A)$, we obtain the contradiction $\varepsilon(A)\leq t < s$.
\end{proof}

The existence problem has been refined to a minimization result.
In particular, the minimal evolution dimension is completely determined by the next formula.

\begin{corollary}
    Let $A$ be a finite-dimensional commutative algebra. Then $$\operatorname{edim}(A)=\dim A+\varepsilon(A).$$
\end{corollary}

\medskip

Theorem~\ref{thm:main} gives an $\frac{n(n+1)}{2}$-dimensional evolution algebra embedding $A$ for every $n$-dimensional commutative algebra $A$. It is therefore natural to ask whether this upper bound is attained.

\begin{question}
    Fixed a field $\mathbb{F}$. {Is there any $n$-dimensional commutative algebra $A$ such that $\operatorname{edim}(A) = \frac{n(n+1)}{2}$?}
\end{question}

\subsection{Properties of evolution envelopes}
After establishing the existence of evolution envelopes and determined their minimal possible dimension, we now study to what extent an evolution envelope reflects structural properties of the original algebra.

Let $A$ be an algebra and denote by $A^{(1)} = A$, $A^{(k+1)} = A^{(k)} A^{(k)}$, for $k\geq 1$, the terms in the derived series.

\begin{proposition}\label{prop:properties_env}
    Let $A$ be a finite-dimensional commutative algebra over a field $\mathbb{F}$. For any evolution algebra $E$ such that  $\Phi: A \hookrightarrow E$ with $E^2\subseteq \Phi(A)$, we have
    \begin{enumerate}
        \item For $k\geq 2$, we have $\Phi(A^{(k)}) \subseteq E^{(k)} \subseteq \Phi(A^{(k-1)})$ and $$\dim A^{(k)} \leq \dim E^{(k)} \leq \dim A^{(k-1)}.$$
        \item If $A$ is perfect (i.e. $A^2=A$), then $A\cong E^2$.
        \item The algebra $A$ is solvable if and only if $E$ is solvable.
        \item The map $\Phi$ induces a bijection between the sets of idempotents of
        $A$ and $E$.
    \end{enumerate}
\end{proposition}
\begin{proof}
    Firstly, we have $\Phi(A^2)\subseteq \Phi(A)^2\subseteq E^2 \subseteq \Phi(A)$, so $\dim A^{(2)} \leq \dim E^{(2)} \leq \dim A^{(1)}$ by the injectivity of $\Phi$. By induction, we can prove that $\Phi(A^{(k)}) \subseteq E^{(k)} \subseteq \Phi(A^{(k-1)})$, and the first claim follows. The second and the third claims are consequences of the first one. 
    
    For the last claim, let $e\in E$ be an idempotent. Then $e\in E^2 \subseteq \Phi(A)$ and $e = \Phi(a)$ for some $a\in A$. Then we have $\Phi(a) = e = e^2 = \Phi(a^2)$ and $a^2 = a$, by the injectivity. Moreover, if $e'\neq e$ is another idempotent with $e' = \Phi(a')$, then we have $a \neq a'$. The converse is clear. 
\end{proof}

The preservation of solvability under evolution envelopes shows that the problem of characterizing solvable evolution algebras already contains the corresponding problem for arbitrary finite-dimensional commutative algebras. In fact, applying any characterization for solvability of evolution algebras to the explicit envelope associated with $A$ in Section~\ref{comideals} would yield a characterization for the solvability of $A$. This observation may help explain the difficulty of obtaining a general characterization of solvability within the class of evolution algebras. In fact, we have the following remark.

\begin{remark}
    Conjecture~3.6 in \cite{GP25} asserts that a complex evolution algebra has 
    no non-zero idempotents if and only if it is solvable. We show that the direct implication fails in general. 
    
    Consider the complex commutative algebra $A$ with basis $\left\{b_1, b_2\right\}$ and multiplication given by 
    $$b_1^2=2b_1+b_2, \quad b_1 b_2 = b_2, \quad b_2^2 = 0.$$
    Then $A$ is a complex $2$-dimensional commutative algebra, and a rescaling of its generators shows that $A$ belongs to the isomorphism class $A_1(1/2)$ described in \cite[Table 1]{2-dim}. Therefore, it does not admit non-zero idempotents by \cite[Corollary 3.2]{2-dim}. Moreover, the multiplication table shows that $A^2=A$, so $A$ is not solvable. We can embed $A$ into an evolution algebra $E$ using either of the constructions given in the proofs of Theorems~\ref{thm:main} or~\ref{thm:embedding_amir}. By Proposition~\ref{prop:properties_env}, the algebra $E$ is not solvable and contains no non-zero idempotents, providing a counterexample to \cite[Conjecture 3.6]{GP25}.

    More precisely, define $v = b_2$ and the linear functional $\beta\colon A\rightarrow \mathbb{C}$ by $\beta(b_i) = 1$. Then we have $b_1b_2 = \beta(b_1)\beta(b_2) v$, as required in the hypotheses of Theorem \ref{thm:embedding_amir}. Let $E$ be the evolution algebra with natural basis $\left\{e_1, e_2, f\right\}$ and multiplication 
    \begin{equation*}
        e_1^2 = 2e_1 + 2f, \quad e_2^2 = -e_2 - f, \quad f^2 = e_2 + f.
    \end{equation*}
    Then the linear map $\Phi\colon A\rightarrow E$, defined by $\Phi(b_1) = e_1 + f$ and $\Phi(b_2) = e_2 + f$, is the algebra monomorphism with $E^2\subseteq \Phi(A)$ given by Theorem \ref{thm:embedding_amir}.
    Hence, the algebra $E$ is a counterexample  to \cite[Conjecture 3.6]{GP25}, and according to \cite{Class2dim} its dimension is minimal among possible counterexamples.  While preparing the final version of this manuscript, the authors became aware of the independent work \cite{HuWen26}, where a three-dimensional counterexample to \cite[Conjecture~3.6]{GP25} is also constructed.
\end{remark}

We now turn to a rank-theoretic description of the same
embedding problem. The key observation is that the multiplication of a commutative algebra is encoded by the quadratic map $H_A(x)=x^2$. Over fields of characteristic different from
two, polarization allows us to pass between these two descriptions and to
reformulate evolution envelopes in terms of simultaneous Waring decompositions.

\section{An interpretation by means of the Waring rank}\label{secwaring}

Let $\mathbb{F}$ be a field of characteristic different from two. Recall the correspondence between homogeneous polynomial maps and commutative algebras via polarization.
If $V=\mathbb{F}^n$, then every $n$-tuple $H = (h_1, \ldots, h_n)$ with $h_i \in \mathbb{F}[x_1, \ldots, x_n]$ defines a polynomial mapping $H:V\rightarrow V$. If $H$ is a homogeneous quadratic $n$-tuple, then $H$ defines a commutative algebra structure on $V$ by
\begin{equation}\label{polarization}
    xy=\frac{1}{2}\big(H(x+y)-H(x)-H(y)\big).
\end{equation}
We denote this algebra by $P_H$, and call it the {\it polarization algebra} of $H$. Conversely, if $A$ is a finite-dimensional commutative algebra and we set $H_A(x)=x^2$, then $P_{H_A} = A$. Here and throughout, when coordinates are used, they are taken with respect to a basis in the context. Thus, there is a correspondence between commutative algebras and homogeneous quadratic maps. 

Evolution algebras correspond to a special class of homogeneous tuples, since a
natural basis eliminates all mixed products. 

\begin{proposition}
    Let $E$ be an $m$-dimensional evolution algebra with a natural basis $B$ and structure matrix $W$ in that basis. Then the associated quadratic map is
$H_E(y_1,\dots,y_m)= W^T (y_1^2, \ldots, y_{m}^2)^T$. Moreover, every homogeneous tuple of this form corresponds to an evolution algebra. 
\end{proposition}

The preceding proposition suggests factoring a general quadratic map through
a diagonal quadratic map. This is the point of a simultaneous Waring
decomposition. Recall that a {\it (simultaneous) Waring decomposition} of $H$ is an expression
\begin{equation}\label{waring}
H(x)=\sum_{k=1}^m \ell_k(x)^2v_k,
\end{equation}
where $\ell_k\in V^*$ and $v_k\in V$. The {\it length} of the decomposition is the number $m$ of summands. The minimal possible length is called the {\it Waring rank} of $H$, which we denote by $\wr\,(H)$. 
Equivalently, the Waring decomposition can be written as $H=D L$, where the linear map 
$L:V\rightarrow\mathbb{F}^m$ is given by $L(x)=(\ell_1(x),\dots,\ell_m(x))$
and
$D:\mathbb{F}^m\rightarrow V$ is given by $D(y_1,\dots,y_m)=\sum_{k=1}^m y_k^2v_k$. The minimal length of a Waring decomposition for which the associated map $L$ is injective will be denoted by $\sr(H)$ and called the {\it injective simultaneous Waring rank}. 

We can now identify injective Waring decompositions with evolution envelopes.

\begin{theorem}
Let $A$ be a finite-dimensional commutative algebra and consider $H:A\rightarrow A$ with $H(x)=x^2$. 
\begin{enumerate}
    \item If $H=DL$ is a Waring decomposition with $L$ injective, then $(P_{\widetilde{H}},L)$ with ${\widetilde{H}} = LD$ is an evolution envelope of $A$. 
    \item Conversely, if $(E,\Phi)$ is an evolution envelope of $A$, then $H = DL$, where $L = \Phi$ and $D= \Phi^{-1} H_E$, is a Waring decomposition with $L$ injective.
\end{enumerate}
\end{theorem}
\begin{proof}
    Suppose first that $H=DL$ is a Waring decomposition with $L$ injective. We prove that $(P_{\widetilde{H}},L)$ with ${\widetilde{H}} = LD$ is an evolution envelope of $P_{H} = A$. 

    First, notice that $P_{\widetilde{H}}$ is an evolution algebra, since for $y \in P_{\widetilde{H}}$ we have 
    $$LD(y) = (\sum_{k=1}^m y_k^2\ell_1(v_k),\dots,\sum_{k=1}^m y_k^2\ell_m(v_k)).$$
    
    Now, the linear map $L: P_{H}\rightarrow P_{\widetilde{H}}$ is injective, so we show that it is a homomorphism. For $a,b\in P_{H}$, we have
    \begin{equation*}
        \begin{split}
            L(ab) &= \frac{1}{2}\big(L H(a+b)-L H(a)- L H(b)\big) \\
            & = \frac{1}{2}\big(L D L(a+b)-L D L(a)- L D L(b)\big)\\
            & = \frac{1}{2}\big({\widetilde{H}} L(a+b)-{\widetilde{H}} L(a)- {\widetilde{H}} L(b)\big) = L(a)L(b).\\
        \end{split}
    \end{equation*}

    Finally, the condition  $P_{\widetilde{H}}^2\subseteq L(P_{H})$ follows from
    $$ ab = \frac{1}{2}\big(\widetilde{H}(a+b)- \widetilde{H}(a)-  \widetilde{H}(b)\big) = L(\frac{1}{2}\big(D(a+b)- D(a)-  D(b)\big)),$$
    for every $a, b\in P_{\widetilde{H}}$, since $\frac{1}{2}\big(D(a+b)- D(a)-  D(b)\big) \in P_H$.

    Conversely, suppose that $(E,\Phi)$ is an evolution envelope of $A$. Let $\widetilde{H}$ be such that $P_{\widetilde{H}} = E$. Write  $L := \Phi$. We can choose $D := L^{-1} \widetilde{H}$ since $\widetilde{H}(x)\in E^2 \subseteq L(A)$. Then we claim that $H = DL$ is a Waring decomposition with $L$ injective and that $\widetilde{H}= LD$. Indeed, for every $x\in A$ we have
    $$DL(x) = L^{-1}\widetilde{H}L(x) = L^{-1}(L(x)L(x)) = L^{-1}L(x^2) = x^2 = H(x).$$
    
    Finally, let $\left\{e_1, \ldots, e_m\right\}$ be a natural basis of $E$ and let $W = (w_{ij})$ be the structure matrix of $E$ with respect to that basis, then we have
    $$D(y_1, \ldots, y_m) = L^{-1}W^T(y_1^2, \ldots, y_m^2)^T$$
    which is of the form $D(y_1,\dots,y_m)=\sum_{k=1}^m y_k^2v_k$ for certain $v_k \in A$. Hence $H=DL$ is a Waring decomposition with $L$ injective as stated.    
\end{proof}

\begin{corollary}\label{corsredim}
    Let $A$ be a finite-dimensional commutative algebra. Then  $$\operatorname{edim}(A)=\sr\,(H_A).$$
    In particular, $A$ is an evolution algebra if and only if $\sr\,(H_A) = \textrm{dim}\, A$. 
\end{corollary}

We denote $\Ann(A) = \left\{x\in A: xA = 0\right\}$ and we write $J$ for the Jacobian matrix.

\begin{proposition}\label{kerann}
    Let $A$ be a finite-dimensional commutative algebra. Given a Waring decomposition $H_A=DL$, then $\ker L \subseteq \Ann(A)$. Moreover, if the length is minimal, then $\ker L = \Ann(A)$.
\end{proposition}
\begin{proof}
    Suppose $z\in \ker L$. We have $2xy = J_{H_A}(x) y$ for $x,y \in A$, see \cite[Lemma 2]{Arenas21}. Since $J_{H_A}(z) = J_{D}(L(z)) J_L(z)$, it follows $2zx = J_{D}(L(z)) J_L(z) x = 0$. Therefore, $z\in \Ann(A)$.
    Moreover, assume the length of the decomposition is minimal. Suppose $z\in \Ann(A)$ and $z\not\in \ker L$. We can assume without any loss of generality that $\ell_1(z)\neq 0$. Then note that for $x\in A$ and $\lambda\in \mathbb{F}$, we have $H(x)=H(x+\lambda z)$. 

    If we put $\lambda= -\ell_1(x)/\ell_1(z)$, then $\ell_1(x+\lambda z)=0$ implying that we can compute $H(x)$ without using the summand $\ell_1(x)^2v_1$. Observe that we have
    $$H(x) = H(x+\lambda z) = \sum_{k=2}^m \ell_k(x+\lambda z)^2 v_k.$$
  This contradicts the minimality of the length. Hence, we have $\ker L = \Ann(A)$.
\end{proof}

\begin{remark}
    The equality does not hold for an arbitrary Waring decomposition. Consider the algebra $A = \mathbb{F}e \oplus \mathbb{F}f$ with multiplication on basis elements given by $e^2 = e$ and zero otherwise. Then $\Ann(A) = \mathbb{F} f$. However, the quadratic map
    $$H_A(x_1, x_2) = (x_1^2, 0)$$
    and the maps 
    $$L(x_1, x_2) = (x_1, x_2, x_2), \qquad D(y_1, y_2, y_3) = (y_1^2 + y_2^2 - y_3^2) e$$ define a Waring decomposition of $H_A$ with $\ker L = 0$.
\end{remark}

\begin{corollary}\label{annwr}
Let $A$ be a finite-dimensional commutative algebra. Then
    $$\sr(H_A) = \wr(H_A) + \dim(\Ann(A)).$$
\end{corollary}
\begin{proof}
     Let $H_A = \sum_{k=1}^m \ell_k(x)^2v_k$ be a Waring decomposition of minimal length. By Proposition~\ref{kerann}, we have $\ker L = \Ann(A)$. Let $c_1, \ldots, c_s$ be a basis of $\ker L$.  
     Choose forms $\ell_{m+k}\in A^*$ such that $\ell_{m+k}(c_j) = \delta_{kj}$ and set $v_{m+k} = 0$, then we have a Waring decomposition $H_A = \sum_{k=1}^{m+s} \ell_k(x)^2v_k$. Moreover, the map $x\to (\ell_1(x),\dots,\ell_{m+s}(x))$ is injective.  Hence, it follows $\sr(H_A) \leq \wr(H_A) + \textrm{dim}\, \Ann(A)$.
     
     Conversely, let $H_A = DL$ be a Waring decomposition with $L$ injective and length $\sr(H_A)$. Then for a non-zero $z\in \Ann(A)$ such that $z\not\in \ker L$, we can use the elimination argument of Proposition~\ref{kerann}. By a straightforward induction, we conclude  $\sr(H_A) \geq \wr(H_A) + \textrm{dim}\, \Ann(A)$.
\end{proof}

By combining Corollary~\ref{corsredim} and Corollary~\ref{annwr}, we obtain the identity
\begin{equation}\label{eq:edim}
\operatorname{edim}(A)
=\dim A+\varepsilon(A)
=\sr(H_A)
=\wr(H_A)+\dim\Ann(A).   
\end{equation}
In particular, if the algebra is unital, then $\operatorname{edim}(A) = \wr(H_A).$ 

\medskip

The correspondence between commutative algebras and homogeneous quadratic maps can be extended to commutative $m$-ary algebras, see \cite{UU19}.

\begin{question}
The notion of an $m$-ary evolution algebra arises naturally as an $m$-ary analogue of an evolution algebra. The arguments developed in this paper seem to extend to the $m$-ary setting, at least over fields of characteristic zero or characteristic $p>m$. In particular, does every finite-dimensional commutative $m$-ary algebra embed into a $m$-ary evolution algebra? We expect the answer to be affirmative. This question provides a natural motivation for the systematic study of $m$-ary evolution algebras. 
\end{question}

\section{The truncated polynomial algebras.} \label{truncated}

    Let $A_n$ denote the algebra of truncated polynomials $A_n:=\mathbb{F}[t]/(t^n).$ In this section, we study the minimal dimensional evolution envelopes of $A_n$ by determining the Waring rank of its squaring map.  

    We first compute the rank of the coordinate functions of the squaring map $H_{A_n}$.

\begin{lemma}\label{lem:wr_f_i}
    Let $\mathbb{F}$ be a field with $\car(\mathbb{F})\ne 2$, and $f_n\in\mathbb{F}[x_0,\ldots,x_n]$ be the quadratic form
    $$f_n=\sum_{k=0}^n x_k x_{n-k}.$$
    Then $\wr(f_n)=n+1$. In particular, the following decompositions are minimal:
    \begin{equation}\label{eq:Waring_dec_truncated_poly}
    f_n =
    \begin{cases}
        \frac{1}{2}\sum_{k=0}^m\Big((x_k+x_{2m+1-k})^2-(x_k-x_{2m+1-k})^2\Big), &\text{ if } n=2m+1.\\
        \\
        x_m^2+\frac{1}{2}\sum_{k=0}^{m-1}\Big((x_k+x_{2m-k})^2-(x_k-x_{2m-k})^2\Big), &\text{ if } n=2m.
    \end{cases}
    \end{equation}
\end{lemma}
\begin{proof}
The identities describing $f_n$ follow immediately from the polarization identity
$$ 2uv=\frac12\big((u+v)^2-(u-v)^2\big).$$
Thus, when $n=2m+1$, the terms in $f_n$ occur in $m+1$ distinct pairs, and the first decomposition has  $2(m+1)=n+1$ square summands. When $n=2m$, there are $m$ such pairs together with the middle term $x_m^2$,
so the second decomposition has $2m+1=n+1$ square summands.

It remains to show that no decomposition with fewer summands is possible.  

Let $M_n$ be the symmetric matrix representing $f_n$, so that
$$
f_n(x)=x M_n x^{\mathsf T},
 \qquad x=(x_0,\ldots,x_n).
$$
Since the coefficient of $x_ix_j$ is zero unless $i+j=n$, the matrix $M_n$ is the antidiagonal matrix
$$
M_n=\begin{pmatrix}
 0&\cdots&0&1\\
 0&\cdots&1&0\\
 \vdots&\iddots&\vdots&\vdots\\
 1&\cdots&0&0
 \end{pmatrix}.
$$
Notice that, for $i\neq j$, the pair of terms $x_ix_j+x_jx_i$ contributes
$2x_ix_j$ to $f_n$, as required. 

The matrix $M_n$ is invertible. In fact,
$M_n^2=\operatorname{Id}_{n+1}$, thus
$$\operatorname{rank}(M_n)=n+1.$$

Now assume that
$$f_n=\sum_{j=1}^{s}\lambda_j\ell_j^2$$
is any Waring decomposition, where $\ell_j(x)=v_jx^{\mathsf T}$ for some vector $v_j\in \mathbb{F}^{n+1}$. 

The symmetric matrix associated with $\lambda_j\ell_j^2$ is $\lambda_jv_j^{\mathsf T}v_j$, which has rank at
most one. Therefore,
$$M_n=\sum_{j=1}^{s}\lambda_jv_j^{\mathsf T}v_j,$$
and the subadditivity of matrix rank gives
$$ n+1=\operatorname{rank}(M_n)
 \leq \sum_{j=1}^{s}\operatorname{rank}
       (\lambda_jv_j^{\mathsf T}v_j)
 \leq s.$$
Therefore every Waring decomposition of $f_n$ contains at least $n+1$ square summands. Since each of the decompositions in Equation \eqref{eq:Waring_dec_truncated_poly} contains
exactly $n+1$ summands, both are minimal.
\end{proof}

The sharper simultaneous Waring decomposition of $H_n{:=}H_{A_n}$ is given in the next proposition.

\begin{proposition}\label{prop:up_wr_H}
    Let $\mathbb{F}$ be a field that contains $2n-1$ pairwise distinct scalars, and let $e_0,\ldots,e_{n-1}$ be the canonical basis of $\mathbb{F}^{n}$. Define
$$
f_i(x_0,\ldots,x_i)=\sum_{r=0}^{i}x_rx_{i-r},\qquad 0\leq i\leq n-1,
$$
and
$$
H_{n}\colon\F^{n}\longrightarrow\F^{n},\qquad
H_{n}(x_0,\ldots,x_{n-1})=\sum_{i=0}^{n-1}f_i(x_0,\ldots,x_i)e_i.
$$
Then $\wr(H_{n})\leq 2n-1$.
\end{proposition}

\begin{proof}
Pick $2n-1$ pairwise distinct scalars $\lambda_1,\dots,\lambda_{2n-1}\in\F$ and set
$$
\ell_j(x)=\sum_{k=0}^{n-1} \lambda_j^kx_k, \qquad (1\leq j \leq 2n-1).
$$
Then 
$$\ell_j(x)^2=\sum_{i=0}^{2n-2}g_i(x)\lambda_j^{i}, \qquad (1\leq j \leq 2n-1).
$$ 
with $g_i(x)=\sum_{k+l=i}x_kx_l$. In particular,  $g_i=f_i$ for $i< n$. 

Since the $(2n-1)\times(2n-1)$ Vandermonde matrix $V=(\lambda_j^{i})$ is invertible, one can solve the equations
$$
\sum_{j=1}^{2n-1} \lambda_j^{i}v_j=
\begin{cases}
    e_i\,,&\text{ for $i< n$}\\
    0\,, &\text{ $n\leq i\le 2n-2$}
\end{cases} 
$$
With those $v_j$, we have
$$
\sum_{j=1}^{2n-1}\ell_j(x)^2v_j
        =\sum_{i=0}^{2n-2}g_i(x)
          \Big(\sum_{j=1}^{2n-1}\lambda_j^iv_j\Big)
        =\sum_{i=0}^{n-1}f_i(x)e_i
        =H_{n}(x).
$$
Hence, we conclude $\wr(H_{n})\leq 2n-1$.
\end{proof}

\begin{example}
If $\car(\F)\ne 2$,  $H_2(x)=x_0^2e_0+2x_0x_1e_1$, and taking
$\lambda_\bullet=\{0,1,-1\}$, 
the argument above leads to  the linear equations
$$\left\{
\begin{aligned}
  v_1 + v_2+ v_3 &= e_0 \\
  v_2 - v_3 &= e_1\\
  v_2 + v_3 &= 0
\end{aligned}
\right.$$
with solutions $v_1=e_0,$ $v_2=\tfrac12 e_1$, and $v_3=-\tfrac12 e_1$. Therefore
$$H_2(x)=x_0^2e_0+\tfrac12\big((x_0+x_1)^2-(x_0-x_1)^2\big)e_1$$ that uses the three squares $x_0^2,(x_0\pm x_1)^2$.    
\end{example}

\begin{remark}
    Observe that 
    $$\max\{\wr(f_i)\,:\, i=0,\ldots,n-1\}\leq\wr(H_{n})\leq \sum_{i=0}^{n-1} \wr(f_i).$$
    Therefore, if $\car(\F)\ne 2$, Lemma \ref{lem:wr_f_i} applies and
     $$n\leq \wr(H_{n})\leq \frac{n(n+1)}{2},$$
     thus $\wr(H_{n})$ is always bounded in this case.

     Nevertheless, the upper bound given in Proposition \ref{prop:up_wr_H} actually requires the base field to have at least $2n-1$ pairwise distinct elements. Indeed, if $\F=\F_3$, and $n=4$, then a brute force calculation shows that $\wr(H_4)>7.$
\end{remark}

To obtain a lower bound, we recall briefly the relation between the Waring rank and the multiplicative complexity of an algebra. Given an associative unital algebra $A$, let $L(A)$ denotes the complexity of $A$ \cite[Definition 3]{Alder-Strassen}. By \cite[Proposition, p.~204]{Alder-Strassen}, we have 
\begin{equation}\label{eq:complexy_vs_waring}
    L(A) \leq \wr(H_A).
\end{equation}
Then, Alder-Strassen's Theorem \cite{Alder-Strassen} gives rise to the following.

\begin{proposition}\label{prop:lower bound}
    Let $\mathbb{F}$ be a field of characteristic different from two. Then $$\wr(H_{A_n})\geq 2 n - 1.$$ 
\end{proposition}
\begin{proof}
    Let $\overline{\F}$ denote the algebraic closure of $\F$, and define $$\overline{A}_n:=A_n\otimes_\F \overline{\F}=\overline{\F}[t]/(t^n).$$

    On one hand, since every Waring decomposition of $H_{A_n}$ induces a Waring decomposition of $H_{\overline{A}_n}$, then 
    \begin{equation}\label{eq:Lower_closure}
    \wr(H_{\overline{A}_n})\le \wr(H_{{A}_n}).    
    \end{equation}

    On the other hand, $\overline{\F}$ is an infinite field, thus by \cite[Theorem, p.\ 203]{Alder-Strassen} we obtain
    \begin{equation}\label{eq:Alder-Strassen}
    L(\overline{A}_n)=2n-1.    
    \end{equation}
    The result follows by combining Equations~\eqref{eq:complexy_vs_waring}, \eqref{eq:Lower_closure}, and \eqref{eq:Alder-Strassen}.
\end{proof}

The simultaneous Waring decomposition given in Proposition \ref{prop:up_wr_H} and the Alder–Strassen lower bound in Proposition \ref{prop:lower bound} completely determine the evolution dimension of the truncated polynomial algebra.

\begin{theorem}\label{trunc}
    Let $\mathbb{F}$ be a field of characteristic different from two that contains $2n-1$ pairwise distinct scalars and let $n\geq 1$. Then $$\operatorname{edim}(A_n) = 2n-1.$$
    In particular, if $n>1$ then $A_n$ is not an evolution algebra.
\end{theorem}
\begin{proof}
    Since $A_n$ is a unital algebra, $\Ann(A_n)=0$ and Equation~\eqref{eq:edim} reduces to $\operatorname{edim}(A_n) = \wr(H_{A_n}).$ Since $H_{A_n}=H_n$, the statement follows by Proposition~\ref{prop:up_wr_H} and Proposition~\ref{prop:lower bound}. 

    Finally, if $n>1$ then $$\operatorname{edim}(A_n) = 2n-1>n=\dim(A_n),$$ hence $A_n$ is not an evolution algebra.
\end{proof}

\begin{remark}
    The minimal evolution envelopes of $A_n$ are not unique up to isomorphism. Note that if $\mathbb{F} = \mathbb{C}$, we can apply the construction of Proposition~\ref{prop:up_wr_H} to $A_2$ using $\lambda_\bullet$ equal to one of the sets
    $$\left\{-1,0,1\right\},\qquad \qquad \left\{1,2,3\right\},$$
    and denote the resulting $3$-dimensional evolution algebra by $E_{\left\{-1,0,1\right\}}$ and $E_{\left\{0,1,2\right\}}$, respectively. A direct computation shows that these algebras are not isomorphic: for example, notice that $$1=\dim \operatorname{Der}\, (E_{\left\{-1,0,1\right\}})\neq \dim \operatorname{Der}\, (E_{\left\{1,2,3\right\}})=0.$$
    Moreover, not even the $\Delta$-traces coincide, in the sense of \cite{CFK24}, since 
    $$\left\{{\bf I}_2,{\bf E}_1\right\} =\Delta\, (E_{\left\{-1,0,1\right\}}) \neq \Delta\, (E_{\left\{1,2,3\right\}}) = \left\{{\bf E}_1,{\bf E}_1,{\bf E}_1\right\}.$$
\end{remark}

\begin{question}
    Evolution algebras trivially have a unique minimal dimensional evolution envelope. Which commutative algebras have a unique, up to isomorphism, minimal evolution envelope?
\end{question}

\section{Embeddings preserving additional structure}\label{addi}

The preceding sections address the existence and minimal dimension of evolution
algebra embeddings. We now consider whether evolution embeddings can be chosen to
preserve additional structure of the original algebra.

\subsection{Embeddings preserving automorphisms} 
Given a finite-dimensional algebra $A$ with finite automorphism group, it is natural to ask if the embedding can be chosen so that every automorphism of $A$ extends to an automorphism of the ambient evolution algebra. The following result shows that this is always possible. 

\begin{theorem}\label{thm:equivariant_embedding}
    Let $A$ be a finite-dimensional commutative $\mathbb{F}$-algebra and suppose that $\Aut(A)$ is finite. Then there exists an algebra monomorphism
$$\widetilde\psi\colon A\longrightarrow \widetilde X$$ such that $\widetilde X$ is a finite-dimensional evolution algebra and every automorphism of $A$ lifts to an automorphism of the evolution algebra $\widetilde X$. More precisely, there exists a group monomorphism
$$
\begin{array}{ccc}
     \Aut(A) &\longrightarrow & \Aut(\widetilde X),\\
     \tau& \longmapsto &\widetilde g_\tau
\end{array}
$$
such that
\begin{equation}\label{eq:equivariance}
 \widetilde g_\tau\circ \widetilde\psi
 =
 \widetilde\psi\circ \tau
\end{equation}
for every $\tau\in \Aut(A)$.
\end{theorem}
\begin{proof}
 To simplify the exposition, we set $G:=\Aut(A)$. 
 Let $\psi_0:A\longrightarrow X_0$ be the algebra monomorphism given by either Theorem \ref{thm:main} or Theorem \ref{thm:embedding_amir}. Now, for every $\sigma\in G$, take a copy $X_\sigma$ of $X_0$, and denote by
$$c_\sigma\colon X_0\longrightarrow X_\sigma$$
the canonical copy isomorphism. 

Define
\begin{equation}\label{eq:X-tilde}
 \widetilde X:=\bigoplus_{\sigma\in G} X_\sigma.
\end{equation}
Since $G$ is finite, $\widetilde X$ is finite-dimensional. Moreover, since a finite direct sum of evolution algebras is again an evolution algebra, $\widetilde X$ is an evolution algebra. A natural basis of $\widetilde X$ is obtained by taking the union of the natural bases of all the summands $X_\sigma$.

For $a\in A$, define
\begin{equation}\label{eq:psi-tilde}
 \widetilde\psi(a)
 :=
 \big(c_\sigma(\psi_0(\sigma^{-1}(a)))\big)_{\sigma\in G}
 \in
 \bigoplus_{\sigma\in G} X_\sigma.
\end{equation}
Equivalently, the $\sigma$-component of $\widetilde\psi(a)$ is
$$
 \widetilde\psi(a)_\sigma
 =
 c_\sigma(\psi_0(\sigma^{-1}(a))).
$$

First, we prove that $\widetilde\psi$ is an algebra homomorphism. For each $\sigma\in G$, the map
$$a\longmapsto c_\sigma(\psi_0(\sigma^{-1}(a)))$$
is a composition of algebra homomorphisms
$$ A\xrightarrow{\sigma^{-1}} A\xrightarrow{\psi_0} X_0\xrightarrow{c_\sigma} X_\sigma.
$$
Therefore each component of $\widetilde\psi$ is an algebra homomorphism. Since multiplication in $\widetilde X=\bigoplus_{\sigma\in G}X_\sigma$ is componentwise, $\widetilde\psi$ itself is an algebra homomorphism.

It is injective because each component
$$a\longmapsto c_\sigma(\psi_0(\sigma^{-1}(a)))$$
is injective.

Now fix $\tau\in G$. Define a linear map
\[
\widetilde g_\tau:\widetilde X\longrightarrow \widetilde X
\]
by
\begin{equation}\label{eq:g-tau}
 (\widetilde g_\tau(x))_\sigma
 =
 c_\sigma c_{\tau^{-1}\sigma}^{-1}(x_{\tau^{-1}\sigma})
\end{equation}
for every $x=(x_\sigma)_{\sigma\in G}\in \widetilde X$ and every $\sigma\in G$.

Thus $\widetilde g_\tau$ simply permutes the direct summands according to the rule
\[
 X_{\tau^{-1}\sigma}\longrightarrow X_\sigma,
\]
using the canonical copy isomorphisms. Hence $\widetilde g_\tau$ is an automorphism of the evolution algebra $\widetilde X$, and $\widetilde g_\tau$ is the identity morphism if and only if $\tau$ is the identity. 

We now verify the equivariance relation. For $a\in A$, the $\sigma$-component of $\widetilde g_\tau(\widetilde\psi(a))$ is
\begin{align*}
 (\widetilde g_\tau(\widetilde\psi(a)))_\sigma
 &=
 c_\sigma c_{\tau^{-1}\sigma}^{-1}
 \left(
 c_{\tau^{-1}\sigma}
 \big(\psi_0((\tau^{-1}\sigma)^{-1}(a))\big)
 \right) \\
 &=
 c_\sigma
 \big(\psi_0((\tau^{-1}\sigma)^{-1}(a))\big).
\end{align*}
Since $(\tau^{-1}\sigma)^{-1}=\sigma^{-1}\tau,$
we obtain
$$ (\widetilde g_\tau(\widetilde\psi(a)))_\sigma
 =
 c_\sigma\big(\psi_0(\sigma^{-1}(\tau(a)))\big).
$$
But this is precisely the $\sigma$-component of $\widetilde\psi(\tau(a))$. Therefore
\[
\widetilde g_\tau(\widetilde\psi(a))
=
\widetilde\psi(\tau(a))
\]
for every $a\in A$, proving \eqref{eq:equivariance}.

Finally, the assignment
$$\tau\longmapsto \widetilde g_\tau$$
is a group homomorphism. Indeed, for $\tau,\rho\in G$ and $x\in \widetilde X$, one checks componentwise that
$$ \widetilde g_\tau\widetilde g_\rho(x)
 =
 \widetilde g_{\tau\rho}(x).$$
The identity element of $G$ maps to the identity automorphism of $\widetilde X$. Hence the desired group homomorphism is obtained.
\end{proof}

\subsection{Embeddings preserving multiplicative bases}
Recall that a basis of an algebra is multiplicative when products of basis
elements are again scalar multiples of basis elements, see \cite{CN16}. The next result shows that this property can be preserved by the ambient evolution algebra, although the multiplicative basis obtained there need not coincide
with a natural basis. 

\begin{theorem}\label{thm:mult2evol}
Let $A$ be an $n$-dimensional commutative algebra over a field $\mathbb{F}$ with
$\car(\mathbb{F})\neq 2$. Assume that $A$ admits a multiplicative
basis. Then there exists an $n^2$-dimensional evolution algebra $B$ which itself admits a multiplicative basis, and a $\mathbb{F}$-algebra morphism $\varphi\colon A\longrightarrow B$ such that
\begin{enumerate}
\item $\varphi$ is a monomorphism, that is, $A$ embeds into $B$,
\item $\varphi(A)$ is an ideal in $B$.
\item $B/\varphi(A)$ is degenerate.
\end{enumerate}
\end{theorem}

\begin{proof}
Choose a multiplicative basis $\mathcal{A}=\{e_1,\dots,e_n\}$ of $A$, and write
$$e_i e_j= c_{i,j}\, e_{\mu(i,j)},$$
where $c_{i,j}=c_{j,i}$, and $\mu\colon\{1,\ldots,n\}\times \{1,\ldots,n\}\to \{1,\ldots,n\}$ is a fixed map with $\mu(i,j)=\mu(j,i).$

Let $B$ be the $n^2$-dimensional $\mathbb{F}$-vector space with basis
$\mathcal B=\{b_{i,r}:1\le i,r\le n\},$ and define a linear monomorphism
$\varphi\colon A\longrightarrow B$
given by $$\varphi(e_i)=\sum_{r=1}^n b_{i,r}+\sum_{s<i}^n b_{s,i} - \sum_{s>i}^n b_{s,i}.$$

We now think of $B$ as an evolution $\mathbb{F}$-algebra with natural basis $\mathcal B$, and multiplication given by:
$$
b_{i,j}^2=
\begin{cases}
 \frac{c_{i,j}}{2} \varphi(e_{\mu(i,j)})=\frac{1}{2}\varphi(e_ie_j) &\text{ if } i<j,\\
c_{i,i}\, \varphi(e_{\mu(i,i)}) = \varphi(e_i^2)  & \text{ if }i=j,\\
\frac{-c_{i,j}}{2} \varphi(e_{\mu(i,j)})=\frac{-1}{2}\varphi(e_ie_j) &\text{ if } i>j.
\end{cases}
$$
Then $\varphi$ becomes an algebra monomorphism. Note that if $i\ne j,$ and assume $i<j$, then 
\begin{align*}
\varphi(e_i)\varphi(e_j) & =\big(\sum_{r=1}^n b_{i,r}+\sum_{s<i}^n b_{s,i} - \sum_{s>i}^n b_{s,i}\big)\big(\sum_{t=1}^n b_{j,r}+\sum_{u<j}^n b_{u,j} - \sum_{u>j}^n b_{u,j}\big)\\
& =(b_{i,j}-b_{j,i})(b_{j,i}+b_{i,j})\\
& = b_{i,j}^2-b_{j,i}^2\\
& = \frac{1}{2}\varphi(e_ie_j) - \frac{-1}{2}\varphi(e_ie_j)\\
& = \varphi(e_ie_j).
\end{align*}
In the case of squares,
\begin{align*}
\varphi(e_i)^2 & =\big(\sum_{r=1}^n b_{i,r}+\sum_{s<i}^n b_{s,i} - \sum_{s>i}^n b_{s,i}\big)^2\\
& =\sum_{r=1}^n b_{i,r}^2+\sum_{s\ne i}^n b_{s,i}^2\\
& = b_{i,i}^2+\sum_{j<i} \big(b_{i,j}^2+b_{j,i}^2\big) +\sum_{j>i} \big(b_{i,j}^2+b_{j,i}^2\big)\\
& = \varphi(e_i^2)+\sum_{j<i} \big(\frac{-1}{2}\varphi(e_ie_j)+\frac{1}{2}\varphi(e_ie_j)\big) +\sum_{j>i} \big(\frac{1}{2}\varphi(e_ie_j)+\frac{-1}{2}\varphi(e_ie_j)\big)\\
& = \varphi(e_i^2).
\end{align*}

So $A$ embeds in the evolution algebra $B$ with natural basis $\mathcal{B}$. But $\mathcal{B}$ is not multiplicative, so it still remains to be proven that the evolution algebra $B$ admits a multiplicative basis.

For each $k=1,\ldots,n$ define
$$\mathcal{B}_k:=\{b_{i,j}\in\mathcal{B}\,:\, \mu(i,j)=k\},$$
and let $B_k$ be the subspace of $B$ spanned by the vectors in $\mathcal{B}_k$, so 
$$B=B_1\oplus\ldots\oplus B_n.$$
Observe that by construction, $B_k^2\leq \mathbb{F}\varphi(e_k)$, and $B_kB_l=0$ if $k\ne l$. 
Therefore, for every $x\in B$, one has $xB_k\le \mathbb{F}\varphi(e_k)$, and
$$
\mathcal{S}:=\varphi(\mathcal{A})\bigsqcup \mathcal{B}
$$
is a multiplicative generating system of $B.$ 

Since $\mathcal{A}$ is a multiplicative basis of $A$, and $\varphi$ is a monomorphism, then the vectors $\varphi(\mathcal{A})$ are linearly independent. Therefore, we can select $\widetilde{\mathcal{B}}$, a linearly independent subset of $\mathcal{S}$ such that 
$$\varphi(\mathcal{A})\subseteq \widetilde{\mathcal{B}} \subseteq \mathcal{S}$$ and it is a multiplicative basis of $B$.

Finally, by construction we get that
$$\varphi(A)B \subseteq B^2 \subseteq \varphi(A)$$
hence $\varphi(A)$ is an ideal in $B$ and $B/\varphi(A)$ is degenerate
\end{proof}

\begin{remark}
If one requires the ambient evolution algebra in Theorem \ref{thm:mult2evol} to possess a basis which is simultaneously natural and multiplicative, the assertion is false in general. For example, consider  $A_n$, the algebra defined in Section~\ref{truncated}. It has basis
$\mathcal{A}_n=\{1,t,\ldots,t^{n-1}\}$, that is a multiplicative basis. The algebra is perfect, commutative, and associative. {Also, $A_n$ is not an evolution algebra by Theorem~\ref{trunc}.} 

Assume $A_n$ is embedded into a $m$-dimensional evolution algebra $E$ having a natural multiplicative basis $\mathcal{E}=\{f_s\,:\, s=1,\ldots,m\}$, satisfying:
$$ f_s^2=\lambda_s f_{\sigma(s)}.$$
Without loss of generality, we may further assume that $m$ is minimal. 

Since
$$\{f_{\sigma(s)}\,:\, \lambda_s\ne 0\}\subset \mathcal{E},$$
is a basis of $E^2$, then $E^2$ is again a finite-dimensional evolution algebra having a natural multiplicative basis, and
$$A_n=A_n^2\subset E^2.$$
Therefore, minimality implies $E=E^2$, that is, $E$ is perfect.

Finally, since $E$ is perfect, every subalgebra of $E$ is again an evolution algebra, see \cite{LP25}, but the image of $A_n$ in $E$ is a subalgebra that is not an evolution algebra, a contradiction.
\end{remark}

    The argument in the previous remark can be stated in the following general way.

    \begin{corollary}
        Let $A$ be a finite-dimensional perfect commutative algebra. If $A$ embeds into a finite-dimensional algebra $E$ admitting a natural multiplicative basis, then $A$ is an evolution algebra.
    \end{corollary}

%%%%%%%%%% BIBLIOGRAFIA EN BIBTEX %%%%%%%%%%%%%%

\bibliographystyle{abbrv}
\bibliography{references}

\begin{thebibliography}{10}

\bibitem{Alder-Strassen}
A.~Alder and V.~Strassen.
\newblock On the algorithmic complexity of associative algebras.
\newblock {\em Theor. Comput. Sci.}, 15:201--211, 1981.

\bibitem{idrees}
I.~Alshatnawi, C.~Costoya, and A.~Viruel.
\newblock On the grothendieck ring of finite-dimensional non-degenerate evolution algebras, 2026.
\newblock arXiv:2607.15154.

\bibitem{Arenas21}
M.~Arenas.
\newblock The space of invariant bilinear forms of the polarization algebra of a polynomial endomorphism: An approach to the problem of {A}lbert.
\newblock {\em J. Algebra}, 573:1--15, 2021.

\bibitem{BCS22}
N.~Boudi, Y.~Cabrera~Casado, and M.~Siles~Molina.
\newblock Natural families in evolution algebras.
\newblock {\em Publ. Mat., Barc.}, 66(1):159--181, 2022.

\bibitem{BMV20}
M.~D. Bustamante, P.~Mellon, and M.~V. Velasco.
\newblock Determining when an algebra is an evolution algebra.
\newblock {\em Mathematics}, 8(8), 2020.

\bibitem{CSV16}
Y.~Cabrera~Casado, M.~Siles~Molina, and M.~V. Velasco.
\newblock Evolution algebras of arbitrary dimension and their decompositions.
\newblock {\em Linear Algebra Appl.}, 495:122--162, 2016.

\bibitem{CFK24}
A.~J. {Calderón Martín}, A.~{Fernández Ouaridi}, and I.~Kaygorodov.
\newblock Non-degenerate evolution algebras.
\newblock {\em J. Pure Appl. Algebra}, 228(6):107594, 2024.

\bibitem{CN16}
A.~J. {Calderón Martín} and F.~J. {Navarro Izquierdo}.
\newblock Arbitrary algebras with a multiplicative basis.
\newblock {\em Linear Algebra Appl.}, 498:106--116, 2016.

\bibitem{Class2dim}
J.~M. Casas, M.~Ladra, B.~A. Omirov, and U.~A. Rozikov.
\newblock On evolution algebras.
\newblock {\em Algebra Colloq.}, 21(2):331--342, 2014.

\bibitem{CLTV22}
C.~Costoya, P.~Ligouras, A.~Tocino, and A.~Viruel.
\newblock Regular evolution algebras are universally finite.
\newblock {\em Proc. Am. Math. Soc.}, 150(3):919--925, 2022.

\bibitem{EL15}
A.~Elduque and A.~Labra.
\newblock Evolution algebras and graphs.
\newblock {\em J. Algebra Appl.}, 14(7):1550103, 10, 2015.

\bibitem{EL16}
A.~Elduque and A.~Labra.
\newblock On nilpotent evolution algebras.
\newblock {\em Linear Algebra Appl.}, 505:11--31, 2016.

\bibitem{EL19}
A.~Elduque and A.~Labra.
\newblock Evolution algebras, automorphisms, and graphs.
\newblock {\em Linear Multilinear Algebra}, 0(0):1--12, 2019.

\bibitem{GP25}
X.~Garc{\'i}a-Mart{\'i}nez and A.~P{\'e}rez-Rodr{\'i}guez.
\newblock A note on complete evolution algebras.
\newblock {\em Arch. Math.}, 2026.

\bibitem{HuWen26}
X.-Y. Hu and R.~Wen.
\newblock Idempotent-free non-solvable evolution algebras over {$\mathbb{C}$}, 2026.
\newblock arXiv 2609.25023.

\bibitem{2-dim}
I.~Kaygorodov and Y.~Volkov.
\newblock The variety of two-dimensional algebras over an algebraically closed field.
\newblock {\em Can. J. Math.}, 71(4):819--842, 2019.

\bibitem{LP25}
M.~Ladra and A.~P\'erez-Rodr{\'\i}guez.
\newblock Regular evolution algebras are closed under subalgebras.
\newblock {\em C. R., Math., Acad. Sci. Paris}, 363:1461--1465, 2025.

\bibitem{chinos}
S.~Sriwongsa and Y.~M. Zou.
\newblock On automorphism groups of idempotent evolution algebras.
\newblock {\em Linear Algebra Appl.}, 641:143--155, 2022.

\bibitem{Tian}
J.~P. Tian.
\newblock {\em Evolution algebras and their applications}, volume 1921 of {\em Lecture Notes in Mathematics}.
\newblock Springer, Berlin, 2008.

\bibitem{TV06}
J.~P. Tian and P.~Vojt\v{e}chovsk\'{y}.
\newblock Mathematical concepts of evolution algebras in non-{M}endelian genetics.
\newblock {\em Quasigroups Relat. Syst.}, 14(1):111--122, 2006.

\bibitem{UU19}
U.~Umirbaev.
\newblock {Polarization algebras and their relations}.
\newblock {\em J. Commut. Algebra}, 11(3):433 -- 451, 2019.

\end{thebibliography}
\end{document}